\documentclass[EJP]{ejpecp} 

\usepackage[T1]{fontenc}
\usepackage[utf8]{inputenc}

\usepackage{enumitem}  
\usepackage{prettyref}
\usepackage{bbm}

\SHORTTITLE{On The Overlap Distribution of Branching Random Walks} 

\TITLE{On The Overlap Distribution of Branching Random Walks
\thanks{This reseach was supported in part by
    by the National Science Foundation and New York University}
}

\AUTHORS{%
Aukosh~Jagannath\footnote{Department of Mathematics, University of Toronto, Canada
\EMAIL{aukosh@cims.nyu.edu}}
}


\AMSSUBJ{60K35; 82B44; 82D30; 82D60} 

\SUBMITTED{April 12, 2016} 
\ACCEPTED{July 10, 2016} 

\VOLUME{0}
\YEAR{2012}
\PAPERNUM{0}
\DOI{vVOL-PID}

\ABSTRACT{
In this paper, we study the overlap distribution and Gibbs measure
of the Branching Random Walk with Gaussian increments on a binary
tree. We first prove that the Branching Random Walk is 1 step Replica
Symmetry Breaking and give a precise form for its overlap distribution,
verifying a prediction of Derrida and Spohn. We then prove that the
Gibbs measure of this system satisfies the Ghirlanda-Guerra identities.
As a consequence, the limiting Gibbs measure has Poisson-Dirichlet
statistics. The main technical result is a proof that the overlap
distribution for the Branching Random Walk is supported on the set
$\{0,1\}$. 
}

\newcommand{\cT}{\mathcal{T}}
\newcommand{\cM}{\mathcal{M}}
\newcommand{\cN}{\mathcal{N}}
\newcommand{\R}{\mathbb{R}}
\newcommand{\Z}{\mathbb{Z}}
\newcommand{\N}{\mathbb{N}}
\newcommand{\prob}{\mathbb{P}}
\newcommand{\E}{\mathbb{E}}
\newcommand{\tensor}{\otimes}
\newcommand{\abs}[1]{\lvert#1\rvert}
\newcommand{\norm}[1]{\lvert\lvert#1\rvert\rvert}
\newcommand{\indicator}[1]{\mathbbm{1}_{#1}}

 \newlist{casenv}{enumerate}{4}
 \setlist[casenv]{leftmargin=*,align=left,widest={iiii}}
 \setlist[casenv,1]{label={{\bf\ Case \arabic*.}},ref=\arabic*}

\newrefformat{prop}{Proposition \ref{#1}}
\newrefformat{cor}{Corollary \ref{#1}}
\newrefformat{sub}{Section \ref{#1}}
\newrefformat{rem}{Remark \ref{#1}}

\begin{document}



\section{Introduction}

In this paper, we study the Branching Random Walk (BRW), or directed
polymer, on a binary tree. To fix notation, let $\cT_{N}$ be the
binary tree of depth $N$ and let $\{g_{v}\}_{v\in\cT_{N}\backslash\emptyset}$
be a collection of i.i.d. standard Gaussian random variables indexed
by this tree without its root. We define the Branching Random Walk
by
\[
H(v)=\sum_{\beta\in p(v)}g_{\beta}
\]
 where $p(v)$ is the root-leaf path to $v$ excluding the root. Viewed
as a gaussian process on $\partial T_{N}$, $(H(v))_{v\in\partial\cT_{N}}$
is centered and has covariance structure
\[
\E H_{N}(v)H_{N}(w)=\abs{v\wedge w},
\]
 where $v\wedge w$ denotes the least common ancestor of $v$ and
$w$ and $\abs{\alpha}$ is the depth of $\alpha\in\cT_{N}$. In particular,
\[
\E H_{N}(v)H_{N}(v)=N.
\]
One can think of the root-leaf paths on the tree as polymer configurations
and $H_{N}$ as an energy. We denote the partition function corresponding
to this polymer model by 
\[
Z_{N}(\beta)=\sum e^{\beta H_{N}(v)}
\]
 and the free energy by 
\[
F_{N}(\beta)=\frac{1}{N}\E\log Z_{N}(\beta).
\]
 If we denote the Gibbs measure by 
\[
G_{N,\beta}(v)=\frac{e^{\beta H_{N}(v)}}{Z}
\]
then this induces a (random) probability measure on the leaves $\Sigma_{N}=\partial\cT_{N}$.
In the following $\left\langle \cdot\right\rangle _{N}$ denotes integration
with respect to this measure or the corresponding product measures.
We will drop the subscript $N$ when it is unambiguous. Finally, let
$R(v,w)=\frac{1}{N}\abs{v\wedge w}$, and let $R_{12}=R(v_{1},v_{2})$,
which we call the overlap between two polymers. An important object
in the study of mean field spin glasses is the (mean) overlap distribution
\begin{equation}
\mu_{N,\beta}(A)=\E G_{N,\beta}^{\tensor2}(R_{12}\in A).\label{eq:mu_N-def}
\end{equation}

The Branching Random Walk, was introduced to the mean field spin glass
community in \cite{DerSpo88}. There, Derrida and Spohn argued that
the statistical physics of this model should be similar to the Random
Energy Model (REM). They predicted that the overlap distribution should
consist of one atom at high temperature and two atoms at low temperature.
In the language of Replica theory it should be Replica Symmetric (RS)
at high temperature and one step Replica Symmetry breaking (1RSB)
at low temperature. Furthermore, they predicted that, as with the
REM, the limiting Gibbs measure of the system should be a Ruelle Probability
Cascade (see the discussion preceding Corollary \ref{cor:ov-rpc-conv} for
a definition). As a consequence, it was suggested \cite{ChauvRou97,DerSpo88}
that the BRW should serve as an intermediate toy model for spin glass
systems, between the REM and the Sherrington-Kirkpatrick (SK) model,
as it is still analytically tractable, while having a key feature
of SK that the REM lacks: a strong local correlation structure. 

The study of replica symmetry breaking in its various forms is the
subject of major research in the mathematical spin glass community.
As such it is of interest to have a few simple, but non-trivial examples
in which Replica Symmetry Breaking can be seen essentially ``by hand''.
In this paper, we give proofs of the predictions of Derrida and Spohn
described above using a combination of arguments that are basic to
both fields. In particular, we avoid the use and analysis of the extremal
process.

We begin first with the study of Replica Symmetry and Replica Symmetry
Breaking. Replica Symmetry above the critical temperature was proved
by Chauvin and Rouault in \cite{ChauvRou97}. Our contribution is
proving Replica Symmetry Breaking below the critical temperature,
and in particular obtaining the mass on the atom at $1$. 
\begin{theorem}
\label{thm:(Derrida-Spohn-conjecture)}Let $\beta_{c}=\sqrt{2\log2}$.
Then 
\begin{equation}
\E G_{N,\beta}^{\tensor2}(R_{12}\in\cdot)\rightarrow\mu_{\beta}(\cdot)=\begin{cases}
\delta_{0} & \beta<\beta_{c}\\
\frac{\beta_{c}}{\beta}\delta_{0}+(1-\frac{\beta_{c}}{\beta})\delta_{1} & \beta\geq\beta_{c}
\end{cases}\label{eq:mubeta-def}
\end{equation}
weakly as measures 
\end{theorem}
We now turn toward the characterization of the Gibbs measure for this
system. Our next result is to prove that the Gibbs measure satisfies
a class of identities called the Approximate Ghirlanda-Guerra Identities
which will imply the Ruelle Probability Cascade Structure described
above. To this end, let $(v_{i})$ be $i.i.d.$ draws from $G_{N,\beta}$,
let $R_{ij}=R(v_{i},v_{j})$ and define $R^{n}=\left(R_{ij}\right)_{i,j\in[n]}$.
The doubly infinite array $R=(R_{ij})_{ij\geq1}$ is called the \emph{overlap
array} corresponding to these draws. 
\begin{theorem}
\label{thm:GGI}
The Branching Random Walk satisfies the Approximate
Ghirlanda\\-Guerra identities. 
That is, if $f$ is a bounded measurable
function $[0,1]^{n^{2}}$ then for every $p$, 
\[
\lim_{N\to\infty}\abs{\E\langle f(R^{n})R_{1,n+1}^{p}\rangle _{G_{N,\beta}^{\tensor n+1}}-\frac{1}{n}\left(\E\langle f(R^{n})\rangle _{G_{N,\beta}^{\tensor n}}\E\langle R_{12}^{p}\rangle _{G_{N,\beta}^{\tensor n}}+\sum_{k=2}^{n}\E\langle f(R^{n})R_{1k}^{p}\rangle _{G_{N,\beta}^{\tensor n}}\right)}=0
\]
 \end{theorem}
\begin{remark}
Note that by \prettyref{thm:(Derrida-Spohn-conjecture)}, only the
case $\beta>\beta_{c}$ is interesting in the above theorem. 
\end{remark}

Our proof is a version of the technique introduced by Bovier and Kurkova
in \cite{BovKurk04-1,BovKurk04-2} (see \cite{Bov12} for a textbook
presentation) and is analogous to \cite{ArgZind14,ArgZind15}. An
immediate consequence of this is that the overlap array distribution
for these systems converges to a Ruelle Probability Cascade, see Corollary
\ref{cor:ov-rpc-conv}. This also implies a mode of convergence
of Gibbs measures and the convergence of the weights of balls in support
a Poisson-Dirichlet process, which was first proved by Barral, Rhodes,
and Vargas in greater generality by different methods \cite{BRV12}.
This is explained in the discussion surrounding Corollary \ref{cor:ov-rpc-conv}. 

For experts in Branching Random Walks, we emphasize here the following
point. Just as in the work of Arguin-Zindy, these methods allow us
to obtain Poisson-Dirichlet statistics for the system without an analysis
of the extremal process. In particular, we can avoid an analysis of
the decoration (see, e.g., \cite{SubZei15,Mad15} for this terminology.
), thereby side-stepping a major technical hurdle. 

The Approximate Ghirlanda Guerra Identities (AGGI) have emerged as
a unifying principle in spin glasses. Due to the characterization-by-invariance
theory \cite{PanchSKBook}, we know that the limiting overlap distribution
is an order parameter for models that satisfy the AGGIs, as originally
predicted in the Replica Theoretic literature \cite{MPV87}. As such,
it has become very important to find models that satisfy these identities
in the limit. This has proved to be very difficult. 

They are known to hold exactly for the generic mixed $p$-spin glass
models \cite{PanchSKBook}, the REM and GREM \cite{BovKurk04-1,BovKurk04-2}.
These ideas have extended to the 2D Gaussian Free Field and a class
of Log-Correlated fields \cite{ArgZind14,ArgZind15}. For many other
models, however, we only know these results in a perturbative sense
\cite{ContStarr13,GhirGuerr98,PanchSKBook,Panch1rsbstruc}. A contribution
of this paper is the observation that the Branching Random Walk falls
in to the class of models for which these identities hold exactly. 

We finally turn to the main technical step involved in the proofs
of the above results. Just as with the REM, both of these predictions
can be shown to follow from standard concentration and integration-by-parts
arguments provided one can show that the model is at most 1RSB and
that the top of the support is at $1$ when it is 1RSB. To our knowledge
this result is thought of as folklore in the Branching Random Walk
community. For example, such a result follows from similar ideas to
those in \cite{ABK13,ArgZind15,ArgZind14,Mad15}. The proof of this
result for similar models can be seen in \cite{DingZeit14} and \cite{ABK13}.
In our setting, this is the content of the following proposition.
\begin{proposition}
\label{prop:overlap-supported-0-1}The mean overlap distribution is
supported on the set $\{0,1\}\subsetneq[0,1]$. That is, for any weak
limit we get that 
\[
\E G_{N,\beta}^{\tensor2}(R_{12}\in\cdot)\rightarrow m\delta_{0}+(1-m)\delta_{1}
\]
for some $m\in[0,1]$. 
\end{proposition}
In the Random Energy Model, this is a consequence of the second moment
method combined with a large deviations estimate. In our setting,
however, this argument breaks down due to the correlation structure
of the Branching Random walk. This is often explained by the seemingly
innocuous observation that the sub-leading correction to the expected
maximum (see the definition of $m_{N}$ in \prettyref{sec:support-is-zero-one})
is of the same order of magnitude as in the REM, but the pre-factor
is $3/2$ as opposed to $1/2$. We point the reader to \cite[Section 3.2]{Arg16}
for a discussion of how this small change is a signature of a profound
structural difference. 

In the study of such models, this issue is dealt with by a truncated
second moment method approach. In our setting, this takes the form
of the tilted barrier estimates of Bramson, see \prettyref{sub:Tilted-Barrier-Estimates}. 

Before turning to the proofs of the above, we make the following remarks.
\begin{remark}
In our setting, one does not need the full power of the Ghirlanda-Guerra
identities to obtain the aforementioned characterization of the Gibbs
measure. In particular, as a consequence of \prettyref{prop:overlap-supported-0-1},
the Approximate Ghirlanda-Guerra Identities are equivalent to an approximate
form of Talagrand's identities \cite{PanchSKBook}, which characterize
the Poisson-Dirichlet process through its moments. This is explained
in more detail in the discussion surrounding Corollary \ref{cor:ov-rpc-conv}.
\end{remark}

\begin{remark}
These arguments do not depend greatly on the Gaussian nature of the
problem. In particular, the main technical tool, Proposition \ref{prop:overlap-supported-0-1},
holds in fairly large generality (see \prettyref{rem:extension-to-other-disorders}).
The remaining results are essentially consequences of the sub-Gaussian
tails of the model and an applications of integration-by-parts. These
results should extend to increments that have sub-Gaussian tails.
For experts, we also note that if the decoration process has enough
moments, the first two results follow by an application of the Bolthausen-Sznitman
invariance (see \cite{BolSzn02}\cite[Sect. 2.2]{PanchSKBook}). As
a study of the extendability of these results are not within the scope
of this paper we do not examine these questions further.
\end{remark}

\ACKNO{The author would like to thank Ofer Zeitouni for many helpful discussions
regarding Branching Random Walks and for a careful reading of an early
draft of this paper. This research was conducted while the author
was supported by an NSF Graduate Research Fellowship DGE-1342536,
and NSF Grants DMS-1209165 and OISE-0730136. Preparation of this work was partially supported by NSF OISE-1604232}

\section{The Support of the Overlap Distribution\label{sec:support-is-zero-one}}

In this section, we will prove Proposition \ref{prop:overlap-supported-0-1},
namely that the support of the overlap distribution is the set $\{0,1\}$.
To this end, we introduce the following notation. Let 
\[
S_{v}(l)=\sum_{\substack{w\in p(v)\\
depth(w)\leq l
}
}g_{w}
\]
denote the walk corresponding to the BRW at vertex $v$. In this notation,
$H_{N}(v)=S_{v}(N)$. Let $S(l)$ denote a random walk with standard
Gaussian increments and let $P^{z}$ denote its law conditioned to
start at $z$.

We think of the collection of walkers $(S_{v})_{v\in\partial T_{N}}$
as a pack of walkers that branch at each time step. We call $M_{N}=\max_{v\in\partial\cT_{N}}S_{v}(N)$
the \emph{leader} of the walkers. With slight---or, depending on your
taste, great---exaggeration, we call $m_{N}=\beta_{c}N-\frac{3}{2\beta_{c}}\log N$.
To justifying this simplification, we remind the reader of the result
from \cite{ABR09} that the family of random variables $\left(M_{N}-m_{N}\right)_{N\geq0}$
is tight. In particular, if $\cM_{K}=\left\{ \abs{M_{N}-m_{N}}\leq K\right\} $,
there is a function $\epsilon_{1}(K)$ with 
\[
\lim_{K\to\infty}\epsilon_{1}(K)=0
\]
such that 
\[
P(\mathcal{M}_{K}^{c})\leq\epsilon_{1}(K).
\]
 Of course this does not necessarily show that $m_{N}$ the true location
of $M_{N}$. The true location will be an order $1$ correction from
this. Finally let $\lambda_{N}=\frac{m_{N}}{N}$, and $\Lambda(x)=\frac{x^{2}}{2}$
. In the following we will drop the subscript $N$ whenever possible
for readability, and we say $f\lesssim_{a}g$ if $f\leq C(a)g$ where
$C$ is a constant that depends at most on $a$. 

This section is organized as follows. First we prove a basic estimate
that will be used through out the section. We then prove the main
estimates required to prove Proposition \ref{prop:overlap-supported-0-1}.
We finally turn to the proof of Proposition \ref{prop:overlap-supported-0-1}.
Before we start we make a brief remark regarding the extension of
these results to more general increments than Gaussian.
\begin{remark}
\label{rem:extension-to-other-disorders}The results of this section
hold in more generality than we study here. For the interested reader,
note that in the following we use sub-Gaussian tails (we believe that
one can relax this, however computing the free energy in this setting
becomes delicate), that the increments are $i.i.d.$ and have support
on $(-1/2,1/2)$, and finite moment generating function and rate functions,
and that the tree is $k$-ary. In this setting, $m_{N}=Nx+O(\log N)$
where $x$uniquely solves $\Lambda^{*}(x)=\log k$ and $x>\E X$,
and $\lambda$ uniquely achieves the equality in $\Lambda^{*}(x)+\Lambda(\lambda)=\lambda x$.
For more on this see, e.g., \cite{BDZ14}. To avoid un-necessary notation
and technicalities, we will stick to the Gaussian case where we have
the self-duality of the moment generating function , $\Lambda=\Lambda^{*}$. 
\end{remark}

\subsection{Tilted Barrier Estimates\label{sub:Tilted-Barrier-Estimates}}

In the following, we will repeatedly use of a class of estimates called
\emph{tilted-barrier estimates.} These estimates are used frequently
in the study of Branching Random Walks and Log-Correlated fields and
were, to our knowledge, introduced by Bramson \cite{Bra78}. The goal
of these estimates is to bound probabilities of the form
\[
P^{z}\left(S(l)\leq\lambda l+K\:\forall l\in[T];S(T)\in\lambda T+[a,b]\right).
\]
We think of the underlying event as follows: there is a random walker,
$S(l)$, which starts at $z,$ and two lines, $\lambda l$ and $\lambda l+K$,
which are barriers. Our goal is to compute the probability that the
walker stays below the farther barrier, $\lambda l+K$, for the duration
of its walk but ends in a window near the nearer barrier, $\lambda l$.

The idea of the estimate is to tilt the law of the walker, $P^{z}$,
to a new measure, $Q^{-z}$, so that under $Q^{-z}$ , the walk $\tilde{S}(t)=\lambda\cdot l-S(t)$
will be centered. The result will then follow by an application of
the ballot theorem applied to $\tilde{S}$. We will make this precise
in \prettyref{prop:overlap-supported-0-1}. Before proving this estimate,
we first prove the relevant ballot theorem.
\begin{lemma}
\label{lem:(Ballot-type-theorem)}(Ballot-type theorem) Let $(S(t))_{t=0}^{n}$
be given by $S(t)=X_{0}+\sum_{i=1}^{t}X_{i}$ be a random walk with
$(X_{i})_{i=1}^{n}$ i.i.d. standard Gaussian and $X_{0}$ is the
starting position. For any $A,B\in\Z$ with $0\leq A<A+1\leq B$,
and $z\geq0$, we have that
\[
P^{z}(S(t)\geq0,0<t<n;S(n)\in[A,B])\lesssim\frac{\max\{z,1\}B(B-A)}{n^{3/2}}.
\]
\end{lemma}
\begin{remark}
This proof is a minor modification of the proof of \cite[Theorem 1]{ABR08}.
This modification is explained, for example, in \cite[Lemma 2.1]{BDZ14}
we include its proof for the convenience of the reader.\end{remark}
\begin{proof}
Let $\tau_{h}=\min\left\{ 0<k\leq n:S(t)<-h\right\} $ with the convention
that if the condition never happens, $\tau_{h}=n$. Define the time
reversed, reflected walk, $S^{r}$, 
\[
\begin{cases}
S^{r}(t)=S^{r}(t-1)-X_{n-(t-1)}\\
S^{r}(0)=0
\end{cases}.
\]
That is, $S^{r}(t)=S(t)-S(n)$. Let $\tau_{h}^{r}$ be the same hitting
time as before for the reversed walk. 

Observe that if $S(t)\geq0$ for $t\in[n]$ and $S(n)\in[A,B]$, we
must have that $\tau_{0}\geq\lfloor\frac{n}{4}\rfloor$ , $\tau_{B}^{r}\geq\lfloor\frac{n}{4}\rfloor$,
and $S(n)\in[A,B]$. This is for the following reasons. The first
condition follows immediately from the positivity. The second condition
follows from the fact that if $\tau_{B}^{r}\leq\lfloor\frac{n}{4}\rfloor<\frac{n}{2}$,
$S(n)$ splits as 
\[
S(n)=S(n-\tau_{B}^{r})-S^{r}(\tau_{B}^{r})\geq-S^{r}(\tau_{B}^{r})>B
\]
by positivity of $S$ which contradicts $S(n)\leq B$. As a result,
since $\{ \tau_{0}\geq\lfloor\frac{n}{4}\rfloor\} ,\{ \tau_{L}^{r}\geq\lfloor\frac{n}{4}\rfloor\},$
and $(X_{i})_{i=\lfloor\frac{n}{4}\rfloor+1}^{\lceil\frac{3n}{4}\rceil}$are
independent, it follows that 
\begin{align*}
P^{z}(S(t)\geq0,0<t<n;&S(n)\in[A,B]) \\
& \leq P^{z}\left(\tau_{0}\geq\lfloor\frac{n}{4}\rfloor,\tau_{B}^{r}\geq\lfloor\frac{n}{4}\rfloor,S(\frac{3n}{4})-S^{r}(\lfloor\frac{n}{4}\rfloor)\in[A,B]\right)\\
 & \leq P^{0}(\tau_{z}\geq\lfloor\frac{n}{4}\rfloor)P^{0}(\tau_{B}\geq\lfloor\frac{n}{4}\rfloor)\sup_{x}P^{0}\left(S(\frac{n}{2})\in x+[A,B]\right)\\
 & \lesssim\frac{\max\{z,1\}}{(n/4)^{1/2}}\frac{\max\{B,1\}}{(n/4)^{1/2}}\frac{(B-A)}{\sqrt{n}}.
\end{align*}
The first inequality follows by set containment. The second inequality
follows by a conditioning argument. The second equality follows from
independence, the fact that $S^{r}(t),t\in[n/4]$ does not depend
on $X_{0}=z$, and symmetry. The last inequality follows from \cite[Lemma 3]{ABR08}
and the fact that the last term is at most $(B-A)/\sqrt{n/2}$ by
\cite[Theorem 2]{ABR08} combined with a union bound.
\end{proof}
We now turn to the proof of the Tilted Barrier Estimate.
\begin{lemma}
\label{lem:(Tilted-Barrier-Estimate)}(Tilted Barrier Estimate) Let
$K,z\geq0$, and let $a$ and $b$ be such that $a+1\leq b\leq K$.
Then we have\textup{
\begin{equation}
P^{z}(S(l)\leq\lambda l+K\:\forall l\in[T];S(T)\in\lambda T+[a,b])\lesssim\frac{e^{\lambda(z-a)}}{2^{-T}}\frac{\max\{K+z,1\}\cdot(K-a)\cdot(b-a)}{T^{3/2}}.\label{eq:TBE}
\end{equation}
}\end{lemma}
\begin{proof}
Define the measure $Q^{-z}$ by the tilting
\[
dQ^{-z}(S)=e^{\lambda(S(T)-z)-T\Lambda(\lambda)}dP^{z}(S).
\]
Observe that under $Q^{-z}(S)$, the walk $\tilde{S}$ has no drift,
i.e., 
\[
\E_{Q^{-z}}\left(\tilde{S}\right)=0,
\]
and starts at $-z$. Note that 
\begin{align*}
P^{z}\left(S(l)\leq\lambda l+K\:\right. & \left.\forall l\in[T];S(T)\in\lambda T+[a,b]\right)\\
=&\, \E_{Q^{-z}}\left(e^{-\lambda(S(T)-z)+T\Lambda(\lambda)}\indicator{}(S(l)\leq\lambda l+K\:\forall l\in[T];S(T)\in\lambda T+[a,b])\right)\\
\leq&\, e^{\lambda(z-a)-\frac{\lambda^{2}}{2}T}Q^{K-z}\left(\tilde{S}\geq0\forall l\in[T];\tilde{S}\in K-[a,b]\right).
\end{align*}
By \prettyref{lem:(Ballot-type-theorem)} and the choice of $\lambda$,
we have that 
\[
P^{z}(S(l)\leq\lambda\cdot l+K\:\forall l\in[T];S(T)\in\lambda T+[a,b])\lesssim2^{-T}e^{\lambda(z-a)}\frac{\max\{K+z,1\}(K-a)(b-a)}{T^{3/2}}.
\]

\end{proof}

\subsection{Applications of Concentration and Tilted Barrier Estimates}

In this subsection, we prove three estimates regarding the probability
of the Branching Random Walker having walkers that behave pathologically.
Before we begin, we remind the reader of the interpretation of $H_{N}$
in terms of the walkers $S_{v}$, and the interpretation of $m_{N}$
as (essentially) the location of the leader $M_{N}=\max_{v\in\partial\cT_{N}}S_{v}$
discussed at the beginning of this section.

In our first estimate, we will show that there is a barrier beyond
which it is unlikely for any walker to cross. In order for this probability
to go to zero, we will need that the barrier drifts off to infinity
logarithmically in $N$. This will follow from the Gaussian tails
of the increments. To make this precise, define the event
\begin{equation}
\Gamma_{L}^{N}=\left\{ \exists v\in\partial\cT_{N}:\exists l\in[N]:S_{v}(l)\geq\lambda l+K\right\} \label{eq:barrier-event}
\end{equation}
which is the event that there is a leaf $v$ whose corresponding walker
$S_{v}$ crosses the barrier $L^{K}$ at some time $l\leq N$. We
then have the following lemma.
\begin{lemma}
\label{lem:barrier-bound} For every $\kappa>\frac{5}{2(\beta_{c}-\frac{3}{2e})}$
, for $N$ sufficiently large and $K\geq\kappa\log N$, 
\[
P(\Gamma_{K}^{N})=o_{N}(1).
\]
\end{lemma}
\begin{proof}
By the union bound and the Gaussian tail inequality, we see that 
\begin{align*}
P(\Gamma_{L}^{N}) & \leq\sum_{l=1}^{N}\exp\left[l\frac{\beta_{c}^{2}}{2}-\frac{1}{2l}\left(\frac{m_{N}}{N}l+K\right)^{2}\right].
\end{align*}
Now, 
\begin{align*}
l\frac{\beta_{c}^{2}}{2}-\frac{1}{2l}\left(\frac{m_{N}}{N}l+K\right)^{2} & \leq\frac{3}{2}\frac{\log N}{N}l-(\beta_{c}-\frac{3}{2\beta_{c}}\frac{\log N}{N})K
\end{align*}
so that 
\[
P(\Gamma_{K}^{N})\leq e^{-(\beta_{c}-\frac{3}{2\beta_{c}}\frac{\log N}{N})K}\sum N^{\frac{3}{2}\frac{l}{N}}\leq N^{-\kappa\left(\beta_{c}-\frac{3}{2\beta_{c}}\frac{\log N}{N}\right)}N^{\frac{5}{2}}.
\]
Thus provided 
\[
\kappa>\frac{5}{2\left(\beta_{c}-\frac{3}{2}\frac{\log N}{N}\right)}
\]
this is $o_{N}(1)$. The result follows by optimizing $\frac{3}{2}\frac{\log N}{N}$
in $N$. 
\end{proof}
The next estimate we will need is a control on the probability that
there is a walker that enters the strip $\lambda l+[0,K]$ on a time
of order $tN$ and ends near the leader at time $N$, i.e. near $m_{N}-x$
(recall the interpretation of $m_{N}$ as the location of the leader
from the beginning of this section). To this end, we define the event
\begin{equation}
\Xi_{\epsilon,K}^{N}=\{ \exists v\in\partial\cT_{N},T\in N[\epsilon,1-\epsilon]:S_{v}(l)\leq\lambda l+K\:\forall l\in[N],S_{v}(T)\geq\lambda T,S_{v}(N)\geq m_{N}-x\} \label{eq:walker-in-strip}
\end{equation}
which is the event that there is a leaf, $v$, whose corresponding
walker, $S_{v}$, enters the window $\lambda l+[0,K]$ on the time
scale of $N[\epsilon,1-\epsilon]$. The probability of this event
is bounded as follows.
\begin{lemma}
\label{lem:stip-bound}For all $x,K\in\Z$, $x>0,K\geq1$, and $\epsilon\in(0,1/2)$,
we have that 
\[
P\left(\Xi_{\epsilon,K}\right)\lesssim_{\epsilon}\frac{K^{4}(x+K)^{2}}{N^{1/2}}e^{\lambda_{N}x}.
\]
\end{lemma}
\begin{proof}
By a union bound 
\begin{align}
P(\Xi_{\epsilon,K}) & \leq2^{N}\sum_{T\geq\epsilon N}^{(1-\epsilon)N}P(S(l)\leq\lambda l+K\:\forall t\in[N],S(T)\geq\lambda T,S(N)\geq m_{N}-x).\label{eq:strip-est-sum}
\end{align}
The summand satisfies
\begin{align}
P(S(l)\leq & \lambda l+K\:\forall t\in[N],S(T)\geq\lambda T,S(N)\geq m_{N}-x)\nonumber \\
 & \leq P\left(S(l)\leq\lambda l+K\:\forall l\in[N];S(T)\in\lambda T+[0,K],S(N)\in m_{N}+[-x,K]\right)\nonumber \\
 & \leq \sum_{i=0}^{K-1}P(S(l)\leq\lambda l+K,\forall l\in[T];S(T)\in\lambda T+[i,i+1])\nonumber \\
 & \quad\cdot\max_{z\in[i,i+1]}P(S(l)\leq\lambda l+K,\forall l\in[T,N];S(N)\in m_{N}+[-x,K]\vert S(T)=\lambda T+z).
 \label{eq:strip-est-mult}
\end{align}
We now compute the multiplicands in the summand. 

The first multiplicand can be controlled by the tilted barrier estimate
\prettyref{eq:TBE} to get

\begin{align}
P^{0}\left(S(l)\leq\lambda l+K\:\forall l\in[T];S(T)\in\lambda T+[i,i+1]\right) & \lesssim2^{-T}e^{-\lambda i}\frac{K\cdot(K-i)}{T^{3/2}}.\label{eq:strip-est-A}
\end{align}
To bound the second multiplicand, first let $T'=N-T$. Observe then
that for all $z\in[0,K-1]$, we have $K-z\geq1$, so that by the titled
barrier estimate yields
\begin{align}
P(S(l)\leq\lambda l+K & \:\forall l\in[T,N];S(N)\in m_{N}+[-x,K]\vert S(T)=\lambda T+z)\nonumber \\
 & =P^{z}\left(S(l)\leq\lambda l+K\:\forall l\in[T'];S(T')\in\lambda T'+[-x,K]\right)\nonumber \\
 & \lesssim2^{-T'}e^{\lambda(x+z)}\frac{(K-z)(x+K)^{2}}{\left(T'\right)^{3/2}}\label{eq:strip-est-B}
\end{align}
Plugging \eqref{eq:strip-est-A}-\eqref{eq:strip-est-B} into \eqref{eq:strip-est-mult}
and plugging this into \eqref{eq:strip-est-sum}, yields the desired
bound.
\end{proof}
Our last estimate (once combined with the above two estimates) shows
that it is unlikely for there to be two walkers who branch on the
genealogical time scale $T$--that is, that here are two $v,w\in\partial\cT_{N}$
with $\abs{v\wedge w}=T$--and both end only order $1$ away from
the leader $m_{N}$ (recall again the interpretation of $m_{N}$ from
the beginning of the section). We make this precise in the following
lemma and corollary. 

First we fix a pair $v,w\in\partial T_{N}$ with $\abs{v\wedge w}=T$
and bound the probability of this pathological event. To this end,
define for such a pair $v,w$ the event
\begin{align}
E_{\epsilon}^{N}(T,K,v,w) & =\left\{ S_{v}(l),S_{w}(l)\leq\lambda l+K\,\forall l\in[N];S_{v}(l),S_{w}(l)\leq\lambda l\,\forall l\in N[\epsilon,(1-\epsilon)];\right.\nonumber \\
 & \qquad\qquad\left.S_{v}(N),S_{w}(N)\in m_{N}+[-x,K]\right\} .\label{eq:E-eps}
\end{align}
This is the event that the walkers corresponding to $v$ and $w$
stay below the barrier $\lambda l+K$ for all time, stay below the
barrier $\lambda l$ on the time scale $tN$, and end in the window
$[m_{N}-x,m_{N}+K]$. We control the probability of this event as
follows
\begin{lemma}
\label{lem:walkers-bound}Let $v,w\in\partial T_{N}$ be such that
$R(v,w)=t$, $T=Nt$ , $T'=N-T$ , and $t'=1-t$. Let $K\geq1$$x>0$.
Then 
\begin{align*}
P(E_{\epsilon}^{N}(T,K)) & \lesssim2^{-(N+T')}\frac{e^{2\lambda_{N}x}}{N^{\frac{3}{2}(1+t)}t^{3/2}(1-t)^{3}}(x+K)^{4}K^{4}
\end{align*}
\end{lemma}
\begin{proof}
Observe that
\begin{align*}
P(E_{\epsilon}^{N}(T,&K,v,w))\\
\leq\sum_{j=0}^{\infty} & P\left(S(l)\leq\lambda l+K\,\forall l\in[T];S_{N}(T)\in\lambda l+[-j-1,-j]\right)\\
  \cdot&\max_{z\in[-j-1,-j]}P(S(l)\leq\lambda l+K\,\forall l\in[T,N];S(N)\in m_{N}+[-x,K]\vert S(T)=\lambda T+z)^{2}
\end{align*}
We bound the two multiplicands by application of the tilted barrier
estimate. Observe that the first multiplicand satisfies

\[
P(S(l)\leq\lambda l+K\:\forall l\in[T];\: S_{N}(T)\in\lambda l+[-j-1,-j])\lesssim2^{-T}e^{\lambda\cdot j}\frac{K(K+j+1)}{T^{3/2}}
\]
and that the second satisfies
\begin{align*}
P(S(l)\leq\lambda l+K\:\forall T\in[T,N]; & S(N)\in m_{N}+[-x,K]\vert S_{v}(T')=\lambda T+z)\\
 & \lesssim2^{-T'}e^{\lambda_{N}(x+z)}\frac{\max\{K+\abs{z},1\}(x+K)^{2}}{T'^{3/2}}.
\end{align*}
Here we used that $K-z\geq0$ (for us $z\leq0$). Combining the results
then yields the desired estimate, namely 
\begin{align*}
P(E_{\epsilon}^{N}(T,K,v,w)) & \lesssim2^{-(N+T')}\frac{e^{2\lambda x}}{N^{\frac{3}{2}(1+t)}t^{3/2}(1-t)^{3}}(x+K)^{4}K^{4}.
\end{align*}

\end{proof}
By an application of the union bound, we see that the previous estimate
implies that it is rare for there to be any pairs of leaves $\abs{v\wedge w}=T$
that have this behavior.
\begin{corollary}
\label{cor:walkers-number-bound}Let $v_{T},w_{T}\in\partial T_{N}$
be a pair satisfying $\abs{v\wedge w}=T$. Under the conditions of
\prettyref{lem:walkers-bound}, we have that for $\epsilon\in(0,\frac{1}{2})$
\[
\sum_{T\geq\epsilon N}^{(1-\epsilon)N}2^{N+T'}P(E_{\epsilon}(T,K,v_{T},w_{T}))\lesssim_{\epsilon}\frac{e^{2\lambda x}}{N^{\frac{1}{2}+\frac{3}{2}\epsilon}}(x+K)^{4}K^{4}.
\]

\end{corollary}

\subsection{Proof of \prettyref{prop:overlap-supported-0-1}}

The proof of \prettyref{prop:overlap-supported-0-1} will now follow
immediately from an application of the following two estimates. The
first estimate says that below the critical temperature, the Gibbs
measure gives no mass to points that more than a large, but order
1, distance from the leader. This follows more or less immediately
from the sub-gaussian tails of the increments and the tightness of
the (centered) leader, $M_{N}-m_{N}$. 
\begin{lemma}
\label{lem:low-points-mass-bound}Let $\beta>\beta_{c}$. Then for
each $x$, 
\[
\lim_{x\to\infty}\limsup_{N\to\infty}\E G_{N}(H\leq m_{N}-x)=0.
\]
\end{lemma}
\begin{proof}
Before we begin, we make the following useful definitions. For readability,
we suppress dependence on $N$ whenever it is unambiguous. Let 
\[
V_{y}^{N}=\left\{ v\in\partial T_{N}:H_{N}(v)\in m_{N}-[y,y+1]\right\} 
\]
and $\cN_{y}^{N}=card(V_{y}^{N})$. Choose $\alpha=\alpha(\beta)$
such that $\beta>\beta_{c}(1+\alpha)$. Let $u_{n}=u_{n}(\alpha)=(1+\alpha)^{-1}N^{1/2}2^{-1/2}$.

Now, recall from \cite[Prop. 3.3]{DingZeit14} that for all $y\geq0$,
\begin{align*}
\E\cN_{y}^{N} & \lesssim Ne^{-\beta_{c}y-\frac{y^{2}}{2N}}
\end{align*}
and there is a universal constant $c_{1}$ such that for all $y,u$
with $0\leq u+y\leq\sqrt{N},$ $u\geq-y$, 
\[
P(\cN_{y}^{N}\geq e^{\beta_{c}(y+u)})\lesssim e^{-\beta_{c}u+C\log_{+}(y+u)}.
\]
Let $x_{0}=x_{0}(\alpha,\beta)$ be such that for all $y\geq x_{0}$,
$c_{1}\log((1+\alpha)y)\leq\frac{\beta\alpha}{2}y$. Then for all
$y\in[x,u_{N}]$, 
\[
P(\cN_{y}^{N}\geq e^{\beta_{c}(1+\alpha)y})\lesssim e^{-\frac{\beta_{c}\alpha}{2}y}.
\]

This implies that

\[
\E G_{N}(H_{N}\leq m_{N}-x)\leq\epsilon_{1}(K)+\E G_{N}(H_{N}\leq m_{N}-x)\indicator{\mathcal{T}_{K}}.
\]
Let 
\[
B_{N}^{\alpha}(x)=\left\{ \exists y\in[x,u_{N}]\cap\N:\cN_{y}^{N}\geq e^{\beta_{c}(1+\alpha)y}\right\} .
\]
Note that on the event $\mathcal{M}_{K}$, we have $Z\geq e^{\beta m_{N}-\beta K}$
from which it follows that 
\begin{align*}
\E G_{N}(H_{N}(v)\leq m_{N}-x)\indicator{\mathcal{T}_{K}} & \leq P(B_{N}^{\alpha}(x))+\E\sum_{y=x}^{\infty}\sum_{H\in V_{y}^{N}}\frac{e^{\beta H}}{Z}\indicator{\mathcal{T}_{K}}\indicator{B_{N}^{\alpha}(x)^{c}}\\
 & \lesssim\sum_{y=x}^{u_{N}}e^{-\frac{\beta_{c}\alpha}{2}y}+e^{\beta K}\left(\sum_{y=x}^{u_{N}}e^{-(\beta-\beta_{c}(1+\alpha))y}+\sum_{u_{N}}^{\infty}Ne^{-(\beta-\beta_{c})y}\right)\\
 & \lesssim_{\alpha,\beta}e^{-c(\alpha)x}+e^{\beta K}\left(e^{-c'(\alpha)x}+Ne^{-c''(\beta)u_{N}}\right).
\end{align*}
Putting these together yields 
\[
\limsup_{N\to\infty}\E G_{N}(H\leq m_{N}-x)\lesssim_{\alpha,\beta}\epsilon_{1}(K)+e^{-c(\alpha)x}+e^{\beta K-c'(\alpha)x}.
\]
Since $K$ was arbitrary, we may take $K=\frac{c'(\alpha)x}{2\beta}$,
and the result follows.
\end{proof}
We now show that it is unlikely to have two points that are both order
1 away from the leader, but overlap strictly in $(0,1)$. The idea
of this estimate is that for this to happen there must be two walkers,
$S_{v_{1}}\text{ and }S_{v_{2}}$, whose branching time is of order
$N$ and both land order $1$ away from the leader. This event is
rare by the above.
\begin{lemma}
\label{lem:high-points-mass-bound}Let $\beta>\beta_{c}$. For all
$x$ and all $A\subsetneq(0,1)$, we have that 
\[
\limsup\E G^{\tensor2}(R_{12}\in A;H(v_{1}),H(v_{2})\geq m_{N}-x)=0.
\]
\end{lemma}
\begin{proof}
Recall the events $\Gamma_{L}$, $\Xi_{L,\epsilon}^{N}$, and $E_{\epsilon}^{N}(T,L,v_{T},w_{T})$
from \prettyref{eq:barrier-event}\prettyref{eq:walker-in-strip},and\prettyref{eq:E-eps}
respectively where $v_{T},w_{T}$ are as per Corollary \ref{cor:walkers-number-bound}.
Observe that 
\begin{align*}
\E G_{N}^{\tensor2}(R_{12}&\in A;\, H_{N}(v_{1}),H_{N}(v_{2})\geq m_{N}-x)\\
 & \leq P(\Gamma_{L}^{N})+P(\Xi_{L,\epsilon}^{N})+\E G_{N}^{\tensor2}(R_{12}\in A;H_{N}(v_{1}),H_{N}(v_{2})\geq m_{N}-x)\indicator{\left(\Gamma_{L}^{N}\right)^{c},\left(\Xi_{L,\epsilon}^{N}\right)^{c}}\\
 & \leq P(\Gamma_{L}^{N})+P(\Xi_{L,\epsilon}^{N})+\E\sum_{R_{12}\in(\epsilon,1-\epsilon)}\indicator{H_{N}(v_{1}),H_{N}(v_{2})\in m_{N}+[-x,K]}\indicator{\left(\Gamma_{L}^{N}\right)^{c},\left(\Xi_{L,\epsilon}^{N}\right)^{c}}\\
 & \leq P(\Gamma_{L}^{N})+P(\Xi_{L,\epsilon}^{N})+\sum_{T\geq\epsilon N}^{(1-\epsilon)N}2^{N+T}P(E_{\epsilon}^{N}(T,L,v_{T},w_{T}))\\
 & \lesssim_{\epsilon}P(\Gamma_{L}^{N})+\frac{L^{4}(x+L)}{N^{1/2}}e^{\lambda_{N}x}+\frac{e^{2\lambda_{N}x}}{N^{\frac{1}{2}+\frac{3}{2}\epsilon}}(x+L)^{4}L^{4}
\end{align*}
where the last inequality follows by \prettyref{lem:stip-bound} and
Corollary \ref{cor:walkers-number-bound}. Choosing $L=c\log N$ for
$c$ large enough, and sending $N\to\infty$ then yields the result
after applying \prettyref{lem:stip-bound}.
\end{proof}
We now prove \emph{\prettyref{prop:overlap-supported-0-1}.}
\begin{proof}
\emph{of \prettyref{prop:overlap-supported-0-1}} By \cite{ChauvRou97},
it suffices to take $\beta>\beta_{c}$. Note that it suffices to show
that $\mu_{N}(\epsilon,1-\epsilon)\to0.$ Let $A=(\epsilon,1-\epsilon)$.
Observe that 
\begin{align*}
\mu_{N}(A) & =\E G_{N}^{\tensor2}\left(R_{12}\in A;H(v_{1})\wedge H(v_{2})\leq m_{N}-x\right)\\
&\qquad \qquad\qquad+E G_{N}^{\tensor2}\left(R_{12}\in A;H(v_{1}),H(v_{2})\geq m_{N}-x\right)\\
 & \leq2\E G_{N}(H_{N}\leq m_{N}-x)+\E G_{N}^{\tensor2}\left(R_{12}\in A;H(v_{1}),H(v_{2})\geq m_{N}-x\right).
\end{align*}
The result then follows from \prettyref{lem:low-points-mass-bound}
and \prettyref{lem:high-points-mass-bound} by taking $N\to\infty$
and then $x\to\infty$.
\end{proof}

\section{The Derrida-Spohn conjecture and the Ghirlanda-Guerra Identities}

In this section we prove the Derrida-Spohn conjecture and show that
the Branching Random Walk satisfies the Ghirlanda-Guerra Identities.

\subsection{Derrida-Spohn Conjecture}

The proof of the Derrida-Spohn conjecture will follow immediately
after the following technical preliminaries. Recall first the following
result of Chauvin and Rouault.
\begin{proposition}
\label{prop:chauv-rou}(Chauvin-Rouault \cite{ChauvRou97}) The free
energy satisfies 
\begin{equation}
F(\beta)=\lim\frac{1}{N}\E\log Z_{N}=\begin{cases}
\log2+\beta^{2}/2 & \beta<\beta_{c}\\
\beta_{c}\beta & \beta\geq\beta_{c}
\end{cases}.\label{eq:CR-FE}
\end{equation}

\end{proposition}
Recall the following integration-by-parts. Let $\Sigma$ denote an
at most countable set; $(x(\sigma))_{\sigma\in\Sigma}$ and $(y(\sigma))_{\sigma\in\Sigma}$
be centered Gaussian processes with uniformly bounded variances and
mutual covariance $C(\sigma^{1},\sigma^{2})=\E x(\sigma^{1})y(\sigma^{2})$
; $G'$ be a finite measure on $\Sigma$; $Z=\sum e^{y(\sigma)}$;
and $G(\sigma)=e^{y(\sigma)}G'(\sigma)/Z$. 
\begin{lemma}
\cite[Lemma 2]{PanchSKBook}\label{lem:GGIBP}(Gibbs-Gaussian Integration-by-parts).
We have the identity 
\[
\E\left\langle x(\sigma)\right\rangle _{G}=\E\left\langle C(\sigma^{1},\sigma^{1})-C(\sigma^{1},\sigma^{2})\right\rangle _{G^{\tensor2}}.
\]
Furthermore, for any bounded measurable $f$ on $\Sigma^{n}$, 

\[
\E\left\langle f(\sigma^{1},\ldots,\sigma^{n})x(\sigma^{1})\right\rangle _{G^{\tensor n}}=\E\left\langle f(\sigma^{1},\ldots,\sigma^{n})\left(\sum_{k=1}^{n}C(\sigma^{1},\sigma^{k})-nC(\sigma^{1},\sigma^{n+1})\right)\right\rangle _{G^{\tensor n+1}}.
\]

\end{lemma}
As a consequence we have the following
\begin{corollary}
\label{cor:diff-fe}We have 
\[
F^{\prime}(\beta)=\beta\int(1-x)d\mu
\]
where $\mu$ is a limit point of the mean overlap measure. \end{corollary}
\begin{proof}
Notice that \prettyref{lem:GGIBP} gives
\[
F_{N}^{\prime}(\beta)=\frac{1}{N}\E\left\langle H_{N}\right\rangle =\beta\frac{1}{N}\E\left\langle NR_{11}-NR_{12}\right\rangle =\beta\E\left\langle 1-R_{12}\right\rangle =\beta\int(1-x)d\mu_{N,\beta}.
\]
Where $\mu_{N,\beta}$ is as per \eqref{eq:mu_N-def}. Since $F_{N}$
and $F$ are convex and $C^{1}$, and $F_{N}\to F$ point-wise on
$\R_{+}$, we have that $F_{N}'\to F'$. This gives us the lefthand
side of the desired equality. Furthermore, since $f(x)=1-x$ is in
$C\left([0,1]\right)$, and the sequence $\mu_{N,\beta}\in\Pr\left([0,1]\right)$
is necessarily tight, weak convergence applied to the last term yields
the righthand side of the desired equality.
\end{proof}
Finally we observe that the limiting overlap distribution has $supp\mu\subset\left\{ 0,1\right\} $.

\subsubsection{Proof of \prettyref{thm:(Derrida-Spohn-conjecture)}}
\begin{proof}
By Corollary \ref{cor:diff-fe} and \prettyref{prop:overlap-supported-0-1},
we know that for any such weak limit, we get 
\[
F^{\prime}(\beta)=\beta\int(1-x)d\mu=\beta m.
\]
Differentiating \prettyref{eq:CR-FE}, equating, and solving for $m$,
we get 
\[
m=\begin{cases}
1 & \beta\leq\beta_{c}\\
\frac{\beta_{c}}{\beta} & \beta\geq\beta_{c}
\end{cases}.
\]

\end{proof}

\subsection{Ghirlanda-Guerra Identities}

We now turn to the proof of the Ghirlanda-Guerra Identities for these
models. We need the following preliminary lemmas. Observe that a standard
application of Gaussian concentration yields the following.
\begin{lemma}
\label{lem:free-en-conc}The free energy concentrates about its mean:
\[
\prob(\abs{F_{N}-\E F_{N}}>\epsilon)\leq2e^{\frac{N\epsilon^{2}}{2\beta^{2}}}
\]

\end{lemma}
As a consequence of this, we find the Gibbs measure concentrates around
a fixed energy level to order $N$.
\begin{lemma}
\label{lem:The-intensive-energy-cocnentrates}The intensive energy
concentrates. In particular,
\[
\lim_{N\to\infty}\frac{1}{N}\E\left\langle \abs{H_{N}-\E\left\langle H_{N}\right\rangle }\right\rangle _{\beta}=0
\]
for each $\beta.$\end{lemma}
\begin{proof}
The result then follows from \prettyref{lem:free-en-conc} after a
modification of the proof of \cite[Theorem 3.8]{PanchSKBook}.
\end{proof}
We now turn to the proof of \prettyref{thm:(Derrida-Spohn-conjecture)}.
\begin{proof}
of \prettyref{thm:(Derrida-Spohn-conjecture)}. Observe first that
for $\beta\leq\beta_{c}$, $\mu=\delta_{0}$ by \prettyref{thm:(Derrida-Spohn-conjecture)}
so that the identities are trivial in this setting. It suffices to
study $\beta>\beta_{c}$. Furthermore, the limiting overlap distribution
is supported on $\{0,1\}$ by \prettyref{prop:overlap-supported-0-1},
so it suffices to show the Approximate Ghirlanda-Guerra Identities
for $p=1$, since $R_{12}^{p}=R_{12}$ when $R_{12}\in\{0,1\}$. The
result then follows by a standard integration-by-parts argument. 

Notice that if we apply \prettyref{lem:GGIBP} with $\Sigma_{N}=\partial T_{N}$,
$x(\sigma)=H_{N}(\sigma)$, $y(\sigma)=\beta H_{N}(\sigma)$ , and
$C(\sigma^{1},\sigma^{2})=\beta NR_{12}$ it follows that 
\[
\frac{1}{N}\E\left\langle f(R^{n})H_{N}(\sigma^{1})\right\rangle =\beta\E\left\langle f(R^{n})\left(\sum_{k=1}^{n}R_{1k}-nR_{1,n+1}\right)\right\rangle 
\]
and 
\[
\frac{1}{\beta N}\E\left\langle H_{N}(\sigma^{1})\right\rangle =\E\left\langle 1-R_{12}\right\rangle .
\]
As a result,

\begin{align*}
\frac{1}{\beta N}\E\left\langle f(R^{n})\left(H_{N}(\sigma^{1})-\right. \right.&\left.\left.\E\left\langle H_{N}(\sigma)\right\rangle \right)\right\rangle \\ &= \frac{1}{\beta N}\left(\E\left\langle f(R^{n})H_{N}(\sigma^{1})\right\rangle -\E\left\langle f(R^{n})\right\rangle \E\left\langle H_{N}(\sigma)\right\rangle \right)\\
& =  \left(\E\left\langle f(R^{n})\right\rangle +\sum_{k=2}^{n}\E\left\langle f(R^{n})R_{1k}\right\rangle -n\E\left\langle f(R^{n})R_{1,n+1}\right\rangle \right)\\
 & \qquad\quad-\left(\E\left\langle f(R^{n})\right\rangle -\E\left\langle f(R^{n})\right\rangle \E\left\langle R_{12}\right\rangle \right)\\
& =  \E\left\langle f(R^{n})\right\rangle \E\left\langle R_{12}\right\rangle +\sum_{k=2}^{n}\E\left\langle f(R^{n})R_{1k}\right\rangle -n\E\left\langle f(R^{n})R_{1,n+1}\right\rangle .
\end{align*}
This implies that 
\begin{align*}
\vert\E\left\langle f(R^{n})\cdot R_{1,n+1}\right\rangle  & -\frac{1}{n}\left(\E\left\langle f(R^{n})\right\rangle \E\left\langle R_{12}\right\rangle +\sum_{k=2}^{n}\E\left\langle f(R^{n})\cdot R_{1k}\right\rangle \right)\vert\\
 & =\frac{1}{\beta nN}\abs{\E\left\langle f(R^{n})\left(H_{N}-\E\left\langle H_{N}\right\rangle \right)\right\rangle }\\
 & \leq\norm{f}_{L^{\infty}([0,1]^{n^{2}})}\frac{1}{\beta nN}\E\left\langle \abs{H_{N}-\E\left\langle H_{N}\right\rangle }\right\rangle \rightarrow0
\end{align*}
by \prettyref{lem:The-intensive-energy-cocnentrates} .
\end{proof}
The Approximate Ghirlanda-Guerra identities have many deep consequences.
We highlight one simple consequence regarding the limit of the overlap
array distribution in these systems. Let $Q_{N}^{\beta}$ be the overlap
array distribution corresponding to $G_{N}^{\beta}$. As $\{Q_{N}\}$
is a sequence of measures on the compact polish space $[0,1]^{\N^{2}}$,
it is tight. Let $Q^{\beta}$ be any limit point of this sequence.
We will show that it is the unique limit point and is given by what
is called a 1RSB Ruelle Probability Cascade which we define presently.

To this end, define $RPC(\zeta)$ for $\zeta=\theta\delta_{0}+(1-\theta)\delta_{1}$
with $\theta\in(0,1]$ as follows. If $\theta\in(0,1)$, let $\left\{ e_{n}\right\} $
be the standard basis for $\ell_{2}$, let $(w_{n})$ be the ranked
points of a Poisson-Dirichlet, $PD(\theta,0)$ process \cite{PitYor97},
and let 
\[
G=\sum w_{n}\delta_{e_{n}}.
\]
If $\theta=1$, let $G=\delta_{0}$. Then $RPC(\zeta)$ is the overlap
array distribution induced by $G$. One important result to note is
that $RPC(\zeta)$ satisfies the Ghirlanda-Guerra identities \cite[Section 2]{PanchSKBook}.
This is a consequence of a standard invariance property of the Poisson-Dirichlet
process/Poisson point processes of Gumbel type. 

What we will show now is that as a consequence of \prettyref{thm:GGI},
we will have that $Q=RPC(\mu_{\beta})$. This an immediate consequence
of the characterization-by-invariance theory used in spin glasses
\cite[Section 15.13]{PanchSKBook,TalBK11Vol2}. In this setting, the
proof is fairly elementary and does not require the full machinery
of this theory. Furthermore, it illustrates some essential ideas for
this method. For these reasons, and to make this presentation self-contained,
we include the proof.
\begin{corollary}
\label{cor:ov-rpc-conv}We have the limit $Q_{N}^{\beta}\stackrel{(d)}{\to}Q^{\beta}$
where $Q^{\beta}$ is the overlap distribution corresponding to $RPC(\mu_{\beta})$
where $\mu_{\beta}$ is as in \eqref{eq:mubeta-def}.\end{corollary}
\begin{proof}
(For readability, we drop the dependence on $\beta$ and denote $RPC(\zeta)$
by $\tilde{Q}$.) It suffices to check that 
\[
Q(R^{n}=A)=\tilde{Q}(R^{n}=A)
\]
for every $n\times n$ Gram matrix, $A$, with entries in $\{0,1\}$.
This is done by a direct computation, which reduces both sides to
the same polynomial in $\mu(R_{12}=1)$. Before, we begin the computation,
observe that $Q$ satisfies:
\begin{itemize}
\item Weak Exchangeability: for every $\pi:\N\to\N$ a permutation of finitely
many indices
\[
Q((R_{ij})_{i,j\geq1}\in A)=Q((R_{\pi(i)\pi(j)})_{i,j\geq1}\in A),
\]

\item Ultrametricity: 
\[
Q(R_{12}\geq\min\{R_{13},R_{23}\})=1,
\]

\item The Ghirlanda-Guerra identities: for every $f$ and $\psi$ bounded
measurable, we have
\[
\int f(R^{n})\psi(R_{1,n+1})dQ=\frac{1}{n}\left[\int fdQ\cdot\int\psi d\mu_{\beta}+\sum_{k=2}^{n}\int f(R^{n})\psi(R_{1,k})dQ\right],
\]

\item The diagonal is almost surely $1$: $Q(R_{11}=0)=0$.
\end{itemize}
The probability $Q(R^{n}=A)$ is computed by recursively applying
the following cases: 
\begin{casenv}
\item $n=2$.

The $Q$-probability of this event is either $\mu_{\beta}(R_{12}=1)$
or $\mu_{\beta}(R_{12}=0)$ .

\item $n\geq3$ and $A$ is not the identity.

By weak exchangeability, we can assume that $A$ is block-diagonal,
and that the blocks are arranged in decreasing size. Take the first
block and suppose that it is of length $m\geq2$. Let $R^{n}(m)$
denote the $n-1\times n-1$ matrix obtained by deleting the $m-$th
row and column of $R^{n}$ and similarly for $A(m)$. By ultrametricity
and the Ghirlanda-Guerra Identities combined with weak exchangeability,
it follows that 
\begin{align*}
Q(R^{n}=A) & =Q(R^{n}(m)=A(m),R_{1m}=1)=\frac{1}{n}Q(R^{n-1}=A(m))\left[\mu(R_{12}=1)+m-1\right]
\end{align*}
It thus suffices to compute $Q(R^{n-1}=A(m))$.

\item $n\geq3$ and $A=Id$

By ultrametricity, it follows that 
\[
Q(R^{n}=Id)=Q(R^{n-1}=Id,R_{1,n}=0)-\sum_{k=2}^{n-1}Q(R^{n-1}=Id,R_{k,n}=1)
\]
To see why, note that if $R^{n-1}=Id$ and $R_{k,n}=1$ for some $k\neq n$,
then the remaining must all be zero. By the Ghirlanda-Guerra identities,
the first event is 
\begin{align*}
Q(R^{n-1}=Id,R_{1,n}=0) & =Q(R^{n-1}=Id)\left[1-\frac{1}{n-1}Q(R_{12}=1)\right].
\end{align*}
The first term is then computed by Case $3$. The second event is
computed by Case $2$.

\end{casenv}
By applying repeatedly these cases, the probability of any such event
is reduced to a polynomial in $\mu(R_{12}=1)$. Evidently, the same
argument applied to $\tilde{Q}$ yields same polynomials, and thus
the desired result.
\end{proof}
We now make the following remarks regarding how this can be understood
to imply certain modes of convergence for the Gibbs measures. 
\begin{remark}
Observe that by the uniqueness portion of the Dovbysh-Sudakov theorem
\cite{Panch10}, this shows that the Gibbs measure $G_{N}$ sampling
converges to $G$ as above in the sense of Austin \cite{Aus15}.
\end{remark}

\begin{remark}
We also observe that for $\beta>\beta_{c}$ we can recover a result
like that from \cite{BRV12} mentioned in the introduction. In particular,
if we partition $\partial T_{N}$ in to groups of leaves with overlap
at least $1-\epsilon$, call them $(B_{i}^{\epsilon})$, then the
ranked weights $(G_{N,\beta}(B_{i}^{\epsilon}))$ converge in law
to the ranked weights of $PD(\frac{\beta_{c}}{\beta},0)$. This follows
from an approximation argument combined with Talagrand's identities
(see, e.g., \cite[Theorem 6.3.5]{Jag14} or \cite[Section 15.4]{TalBK11Vol2}).
\end{remark}


\providecommand{\bysame}{\leavevmode\hbox to3em{\hrulefill}\thinspace}
\providecommand{\MR}{\relax\ifhmode\unskip\space\fi MR }
\providecommand{\MRhref}[2]{%
  \href{http://www.ams.org/mathscinet-getitem?mr=#1}{#2}
}
\providecommand{\href}[2]{#2}
\begin{thebibliography}{}

\bibitem{ABR08}
L.~{Addario-Berry} and B.~A. {Reed}, \emph{{Ballot theorems for random walks
  with finite variance}}, http://www.problab.ca/louigi/papers/ballot.pdf
  (2008).

\bibitem{ABR09}
{\sc Addario-Berry, L.} {\sc and} {\sc Reed, B.} (2009).
\newblock Minima in branching random walks.
\newblock {\em Ann. Probab.\/}~\textbf{37},~3, 1044--1079.
\MR{2537549}


\bibitem{Arg16}
Louis-Pierr {Arguin}, \emph{{Extrema of log-correlated random variables:
  Principles and Examples}}, \ARXIV{math.PR/1601.00582} (2016).

\bibitem{ABK13}
{\sc Arguin, L.-P.}, {\sc Bovier, A.}, {\sc and} {\sc Kistler, N.} (2013).
\newblock The extremal process of branching {B}rownian motion.
\newblock {\em Probab. Theory Related Fields\/}~\textbf{157},~3-4, 535--574.
\MR{3129797}


\bibitem{ArgZind14}
{\sc Arguin, L.-P.} {\sc and} {\sc Zindy, O.} (2014).
\newblock Poisson-{D}irichlet statistics for the extremes of a log-correlated
  {G}aussian field.
\newblock {\em Ann. Appl. Probab.\/}~\textbf{24},~4, 1446--1481.
\MR{3211001}


\bibitem{ArgZind15}
{\sc Arguin, L.-P.} {\sc and} {\sc Zindy, O.} (2015).
\newblock Poisson-{D}irichlet statistics for the extremes of the
  two-dimensional discrete {G}aussian free field.
\newblock {\em Electron. J. Probab.\/}~{\em 20}, no. 59, 19.
\MR{3354619}


\bibitem{Aus15}
{\sc Austin, T.} (2015).
\newblock Exchangeable random measures.
\newblock {\em Ann. Inst. Henri Poincar\'e Probab. Stat.\/}~\textbf{51},~3,
  842--861.
\MR{3365963}


\bibitem{BRV12}
{\sc Barral, J.}, {\sc Rhodes, R.}, {\sc and} {\sc Vargas, V.} (2012).
\newblock Limiting laws of supercritical branching random walks.
\newblock {\em C. R. Math. Acad. Sci. Paris\/}~\textbf{350},~9-10, 535--538.
\MR{2929063}


\bibitem{BolSzn02}
{\sc Bolthausen, E.} {\sc and} {\sc Sznitman, A.-S.} (2002).
\newblock {\em Ten lectures on random media}. DMV Seminar, Vol.~\textbf{32}.
\newblock Birkh\"auser Verlag, Basel.
\MR{1890289}


\bibitem{Bov12}
Anton Bovier, \emph{Statistical mechanics of disordered systems}, Cambridge,
  2012.

\bibitem{BovKurk04-1}
{\sc Bovier, A.} {\sc and} {\sc Kurkova, I.} (2004).
\newblock Derrida's generalised random energy models. {I}. {M}odels with
  finitely many hierarchies.
\newblock {\em Ann. Inst. H. Poincar\'e Probab. Statist.\/}~\textbf{40},~4,
  439--480.
\MR{2070334}


\bibitem{BovKurk04-2}
{\sc Bovier, A.} {\sc and} {\sc Kurkova, I.} (2004).
\newblock Derrida's generalized random energy models. {II}. {M}odels with
  continuous hierarchies.
\newblock {\em Ann. Inst. H. Poincar\'e Probab. Statist.\/}~\textbf{40},~4,
  481--495.
\MR{2070335}


\bibitem{BDZ14}
Maury {Bramson}, Jian {Ding}, and Ofer {Zeitouni}, \emph{{Convergence in law of
  the maximum of nonlattice branching random walk}}, Ann. Inst. H. Poincar\'e
  Probab. Statist. (to appear).\ARXIV{math.PR/1404.3423}

\bibitem{Bra78}
{\sc Bramson, M.~D.} (1978).
\newblock Maximal displacement of branching {B}rownian motion.
\newblock {\em Comm. Pure Appl. Math.\/}~\textbf{31},~5, 531--581.
\MR{0494541}

\bibitem{ChauvRou97}
{\sc Chauvin, B.} {\sc and} {\sc Rouault, A.} (1997).
\newblock Boltzmann-{G}ibbs weights in the branching random walk.
\newblock In {\em Classical and modern branching processes ({M}inneapolis,
  {MN}, 1994)}. IMA Vol. Math. Appl., Vol.~\textbf{84}. Springer, New York,
  41--50.
\MR{1601693}


\bibitem{ContStarr13}
{\sc Contucci, P.}, {\sc Mingione, E.}, {\sc and} {\sc Starr, S.} (2013).
\newblock Factorization properties in {$d$}-dimensional spin glasses.
  {R}igorous results and some perspectives.
\newblock {\em J. Stat. Phys.\/}~\textbf{151},~5, 809--829.
\MR{3055377}


\bibitem{DerSpo88}
{\sc Derrida, B.} {\sc and} {\sc Spohn, H.} (1988).
\newblock Polymers on disordered trees, spin glasses, and traveling waves.
\newblock {\em J. Statist. Phys.\/}~\textbf{51},~5-6, 817--840.
\newblock New directions in statistical mechanics (Santa Barbara, CA, 1987).
\MR{971033}


\bibitem{DingZeit14}
{\sc Ding, J.} {\sc and} {\sc Zeitouni, O.} (2014).
\newblock Extreme values for two-dimensional discrete {G}aussian free field.
\newblock {\em Ann. Probab.\/}~\textbf{42},~4, 1480--1515.
\MR{3262484}


\bibitem{GhirGuerr98}
{\sc Ghirlanda, S.} {\sc and} {\sc Guerra, F.} (1998).
\newblock General properties of overlap probability distributions in disordered
  spin systems. {T}owards {P}arisi ultrametricity.
\newblock {\em J. Phys. A\/}~\textbf{31},~46, 9149--9155.
\MR{1662161}


\bibitem{Jag14}
Aukosh {Jagannath}, \emph{{Approximate Ultrametricity for Random Measures and
  Applications to Spin Glasses}},  \ARXIV{math.PR/1412.7076} (2014).

\bibitem{Mad15}
Thomas Madaule, \emph{Convergence in law for the branching random walk seen
  from its tip}, Journal of Theoretical Probability (2015), 1--37 (English).

\bibitem{MPV87}
{\sc M{\'e}zard, M.}, {\sc Parisi, G.}, {\sc and} {\sc Virasoro, M.~A.} (1987).
\newblock {\em Spin glass theory and beyond}. World Scientific Lecture Notes in
  Physics, Vol.~\textbf{9}.
\newblock World Scientific Publishing Co., Inc., Teaneck, NJ.
\MR{1026102}

\bibitem{Panch10}
{\sc Panchenko, D.} (2010).
\newblock On the {D}ovbysh-{S}udakov representation result.
\newblock {\em Electron. Commun. Probab.\/}~{\em 15}, 330--338.
\MR{2679002}

\bibitem{PanchSKBook}
{\sc Panchenko, D.} (2013).
\newblock {\em The {S}herrington-{K}irkpatrick model}.
\newblock Springer Monographs in Mathematics. Springer, New York.
\MR{3052333}

\bibitem{Panch1rsbstruc}
{\sc Panchenko, D.} (2014).
\newblock Structure of 1-{RSB} asymptotic {G}ibbs measures in the diluted
  {$p$}-spin models.
\newblock {\em J. Stat. Phys.\/}~\textbf{155},~1, 1--22.
\MR{3180967}

\bibitem{PitYor97}
{\sc Pitman, J.} {\sc and} {\sc Yor, M.} (1997).
\newblock The two-parameter {P}oisson-{D}irichlet distribution derived from a
  stable subordinator.
\newblock {\em Ann. Probab.\/}~\textbf{25},~2, 855--900.
\MR{1434129}

\bibitem{SubZei15}
{\sc Subag, E.} {\sc and} {\sc Zeitouni, O.} (2015).
\newblock Freezing and decorated {P}oisson point processes.
\newblock {\em Comm. Math. Phys.\/}~\textbf{337},~1, 55--92.
\MR{3324155}

\bibitem{TalBK11Vol2}
{\sc Talagrand, M.} (2011).
\newblock {\em Mean field models for spin glasses. {V}olume {II}}. Ergebnisse
  der Mathematik und ihrer Grenzgebiete. 3. Folge. A Series of Modern Surveys
  in Mathematics [Results in Mathematics and Related Areas. 3rd Series. A
  Series of Modern Surveys in Mathematics], Vol.~\textbf{55}.
\newblock Springer, Heidelberg.
\newblock Advanced replica-symmetry and low temperature.
\MR{3024566}

\end{thebibliography}
\end{document}